\documentclass[12pt]{amsart}
\usepackage{amsmath, amssymb, amsthm}
\usepackage{enumerate}

\usepackage{fancyhdr}

\makeatletter
\@namedef{subjclassname@2010}{%
\textup{2010} Mathematics Subject Classification}
\makeatother

\newtheorem{theorem}{Theorem}[section]
\newtheorem{lemma}{Lemma}[section]
\newtheorem{proposition}{Proposition}[section]

\theoremstyle{definition}
\newtheorem{defin}{Definition}[section]

\numberwithin{equation}{section}

\newcommand{\R}{{\mathbb{R}}}

\begin{document}
\baselineskip=17pt

\title{Medians, continuity, and vanishing oscillation}
\author[J. Poelhuis]{Jonathan  Poelhuis}
\address{Department of Mathematics\\ Indiana University\\
Bloomington, IN, 47405, USA}
\email{jpoelhui@umail.iu.edu}

\author[A. Torchinsky]{Alberto Torchinsky}
\address{Department of Mathematics\\ Indiana University\\
Bloomington, IN, 47405, USA}
\email{torchins@indiana.edu}

\date{}

\begin{abstract}  In this paper we consider  properties of medians as they pertain to the continuity and vanishing oscillation of a
function. Our approach is based on  the observation  that medians are  related to
local sharp maximal functions restricted to a cube of
$\R^n$.
\end{abstract}


\maketitle

In considering the problem of  the resistance of
materials to certain types of deformations,   F. John was
led to the study of  quasi-isometric   mappings. The
setting is essentially as follows. Let $Q_0\subset \R^n $ be a cube and $f $ a continuous
function on $Q_0$. Assume that to each subcube $Q$ of $Q_0$ with sides parallel to
those of $Q_0$ there is assigned a constant $c_Q$ and let $\mu_Q$ be the function
of the real variable $M$ given by
\[\mu_Q(M)=\frac{|\{y\in Q: |f(y)-c_Q|>M\}|}{|Q|}\,.\]
Let $\phi(M)=\sup_{Q\subset Q_0} \mu_Q(M)$, $0<s<1/2$, and $\lambda$ a number such that
$\phi(\lambda)\le s$. Then under these
assumptions
\[\phi(M)\le A e^{-BM/\lambda} \]
holds for all nonnegative $M$ where $A,B$ are universal functions of $s$ and the dimension $n$.
Thus the space  of functions of bounded mean oscillation $(BMO)$ was introduced  and the
John-Nirenberg inequality established \cite{J}.

Str\"{o}mberg adopted this setting when studying   spaces close to $BMO$ and discussed different ways
of describing the oscillation of a function $f$  on cubes.
He also incorporated the   value $s=1/2$ above, which corresponds to the notion of
  median value $m_f (Q)=m_f(1/2,Q) $ of $f$ over $Q$   \cite{JOS}.
 The local maximal functions of Str\"{o}mberg
are of particular interest because they allow for pointwise estimates for   Calder\'on-Zygmund
singular integral operators \cite{JT, L}.

In this paper we consider  properties of medians as they pertain to the continuity and vanishing oscillation of a
function. Our approach is based on  the observation  that medians are related
to local sharp maximal functions restricted to a cube
$Q_0\subset \R^n$ with parameter $0<s\le 1/2$ by  means of the expression
\begin{align*} M^{\sharp,{\phi}}_{0,s,Q_0 }f(x) &= \sup_{x\in Q,Q \subset Q_0} \inf_c \frac{m_{|f - c|}(1-s,Q)}{\phi(|Q|)}\\
&\sim \sup_{x\in Q, Q \subset Q_0} \frac{m_{|f - m_f(1-s,Q)|}(1-s,Q)}{\phi(|Q|)}\,,
\end{align*}
where $\phi=1$ in the case of functions with bounded median oscillation
with parameter $s$ $(bmo_s)$, and satisfies appropriate
conditions in the case of functions with vanishing median oscillation
with parameter $s$ $(vmo_s)$.
Str\"{o}mberg showed that $ {bmo}_s=BMO$,
and, similarly,
 we show here that for sufficiently small $s$,
$ vmo_s=VMO$, the space of functions of vanishing mean oscillation.
Moreover,
since $VMO$ is  known to contain bounded discontinuous functions \cite{S, SP}, we  complete the
picture  by giving  criteria
for continuity  of functions equivalent to a bounded function on a cube in terms of medians.

The paper is organized as follows. In Section 1 we introduce the notions of median  and maximal median
with respect to a parameter $0<s<1$; when $s\ne 1/2$ we refer to these medians as
 biased   with parameter $s$.
In Section 2 we consider the a.e.\! convergence of maximal biased medians  in the spirit
of Fujii's results for $s=1/2$ \cite{F}. In Section 3, motivated by similar results involving averages \cite{ST},
 we characterize continuity in terms of maximal biased medians with parameter $>1/2$.
In Section 4 we extend the  Str\"{o}mberg decomposition of cubes
  to parameters $ >1/2$. Finally, in Section 5 we consider the spaces of functions
with vanishing median oscillation, establish
a John-Nirenberg type inequality they satisfy, and show that, as anticipated,
they coincide
with $VMO$ for sufficiently small $s$.

\section{Medians and maximal medians}

In what follows we restrict our attention to cubes with sides parallel to the coordinate axis.

\begin{defin}  For a cube $Q\subset \R^n$, $0<s<1$,  and a real-valued measurable function $f$  on $Q$ we say that
  $m_f (s,Q) $ is a \emph{ median value of $f$ over $Q$ with parameter $s$} if
\begin{equation}   |\{y \in Q : f(y)< m_f (s,Q) \}|\le s|Q| \end{equation}
and
\begin{equation}   |\{y \in Q : f(y)> m_f (s,Q) \}|\le (1-s) |Q|\,.\end{equation}
\end{defin}

When $s=1/2$, $m_f (1/2,Q)=m_f(Q)$ corresponds to a median value of  $f$ over $Q$.
The set of median values   of $f$ is one
point or a closed interval as the  example  $f =\chi_{[ 1/2 ,1)}$ on  $ [0,1]$
shows.
  It is therefore  convenient to work with   maximal medians, which are uniquely defined
 \cite{C,F}. More precisely, we have

\begin{defin} For a cube $Q\subset \R^n$, $0<s<1$,
and a real-valued measurable function $f$  on $Q$,
we say that   $M_f(s,Q)$ is the \emph{ maximal median    of $f$ over $Q$ with parameter $s$} if
\[M_f (s, Q) = \sup\{M : |\{ y \in Q : f(y) < M\}|\}\le s|Q|\,.\]
\end{defin}

The reader will have no difficulty in proving  the sup above is assumed, that is to say,
\[ |\{ y \in Q : f(y) < M_f(s,Q)\}|\le s |Q|\,.\]

To justify the nomenclature of maximal median  we verify that    $M_f (s,Q)$ satisfies the conditions
that characterize medians.
(1.1)  is guaranteed
since $M_f (s,Q)$ is the maximum value for which it holds. As for (1.2),
let $B_n= \{y\in Q : f(y) \ge M_f (s, Q) + 1/n\}$ and note that $|B_n|\le  (1-s) |Q|$, all $n$,
and
$\{y\in Q : f(y) > M_f(s,Q)\} \subset \liminf_n B_n$. Then
$|\{y\in Q : f(y) > M_f (s, Q) \}| \le \liminf_n |B_n|\le  (1-s)|Q|$.

Hereafter when considering a median  we   mean the maximal median and denote it
simply by $m_f (s, Q)$.
 Clearly maximal medians satisfy
\begin{equation} |\{y \in Q : f(y) \le m_f (s,Q)\}|\ge s|Q|\,,
\end{equation}
and
\begin{equation}
 |\{y \in Q : f(y) \ge m_f (s,Q)\}|\ge (1-s)|Q|\,.
 \end{equation}

We summarize the basic properties of maximal medians that
are of interest to us in the following Proposition.

\begin{proposition}  Let $Q\subset \R^n$ be a cube, $0<s,t<1$,
and $f,g$ real-valued measurable function  on $Q$.
Then the following properties hold:
\begin{enumerate}
[ \upshape (i)]
\item  For $s < t$,
\begin{equation} m_f (s,Q) \le  m_f (t,Q)\,.
\end{equation}
\item
If $f \le  g$ a.e., then
\begin{equation}
m_f (s,Q) \le m_g(s,Q)\,.
\end{equation}
 \item
For a constant $c$,
\begin{equation} m_f (s,Q) - c = m_{f-c}(s,Q)\,.\end{equation}
\item  If $ f, g \ge 0$ a.e., $0 < s, s_1 < 1$, and $0< t < s+s_1- 1$,
\begin{equation} m_{f+g}(t,Q) \le  m_f (s,Q) + m_g(s_1,Q)\,.\end{equation}
\item
 In general,  $m_{-|f|}(s, Q)\le m_f(s,Q)\le m_{|f|}(s,Q)$. And,
if $m_f (s,Q) \le  0$,
\begin{equation} |m_f (s,Q)|\le m_{|f|}(1-s,Q)\,.\end{equation}
Thus for general $f$,
\begin{equation}
|m_f (s,Q)|\le m_{|f|}(s,Q)\,,\quad  1/2\le s<1\,.
\end{equation}
\item
If $f\ge 0$ is locally integrable  and $f_Q$ denotes the average of $f$ over $Q$, then
\begin{equation}
m_{f}(s, Q)\le \frac1{1-s}\,f_Q\,.
\end{equation}
\end{enumerate}
\end{proposition}

\begin{proof}  (\upshape i) Since $|\{y \in Q : f(y) < m_f (s,Q)\}|\le s|Q| <t|Q|$,
(1.5) holds.

(ii) Up to a set of measure zero $\{y\in Q:   g(y) < m_f (s,Q)\}\subset \{y\in Q: f(y) < m_f (s,Q)\}$.
Therefore $|\{y\in Q:   g(y) < m_f (s,Q)\}|\le  s| Q|$, and so $ m_f (s,Q)\le  m_g (s,Q)$.

(iii) Since  $\{y\in Q:   f(y) < m_f (s,Q)\}= \{y\in Q: f(y)-c < m_f (s,Q)-c\}$
 it readily follows that   $ m_f (s,Q)-c \le m_{f-c}(s,Q)$.
  And since  $\{y\in Q: f(y)-c < m_{f-c} (s,Q)\}=  \{y\in Q: f(y) < m_{f-c} (s,Q)+c\}$,
 $  m_{f-c} (s,Q)+c\le    m_f (s,Q)$. Note that, in particular,  $ m_c (s,Q)=c$.

(iv)  For the sake of argument suppose that $f,g$ are measurable functions on  $Q$
such that $m_{f+g}(t,Q)-( m_f(s,Q)+m_g(s_1,Q))> 2\eta>0$.
Then $\{y \in Q : f(y) < m_f (s,Q)+\eta\}\cap \{y\in Q: g(y)< m_g(s_1,Q)+\eta\}\subset
\{y\in Q: f(y)+g(y)< m_{f+g}( t,Q)\}$. Now, since
$|\{y \in Q : f(y) < m_f (s,Q)+\eta\}|\ge s|Q|$, $| \{y\in Q: g(y)< m_g(s_1,Q)+\eta\}|\ge s_1|Q|$,
and $|\{y \in Q : f(y)+g(y) < m_f (t,Q)\}|\le t|Q|$, it readily follows that
$  s+s_1\le 1+t$, which is not the case.

(v) Since $-|f|\le f \le |f|$, by (1.6),  $m_{-|f|}(s, Q)\le m_f(s,Q)\le m_{|f|}(s,Q)$. Now,
if $m_f (s,Q) \le  0$ note that  $\{y\in Q:  | f(y)| <- m_f (s,Q)\}= \{y\in Q:    m_f (s,Q)<-|f(y)|\}\subset
\{y\in Q:    m_f (s,Q)<f(y)\}$, and, therefore, $|\{y\in Q:  | f(y)| <- m_f (s,Q)\}|\le
(1-s)|Q|$. Consequently,  $ |m_f (s,Q)|\le m_{|f|}(1-s,Q)$.
And, since for
$s\ge 1/2$, $(1-s)\le s$, by (1.5)    we have
$| m_{ f} (s,Q)|\le m_{|f|}(s,Q)$   for that range of $s$.

(vi) We may assume that $m_{f}(s, Q)\ne 0$.  Then by (1.4) and
Chebychev's inequality,
\[ (1-s)|Q|\le |\{y \in Q : f(y) \ge m_f (s,Q)\}|\le\frac{1}{ m_f (s,Q)}\,\int_Q  f(y) \,dy\,,
\]
and the conclusion follows.
\end{proof}

The restriction   $1/2\le s<1$  is necessary for  (1.10) to hold.
Let $ Q = [0, 1]$  and $f(x) = -2\chi_{[0, 1/2 )}(x) + \chi_{[ 1/2 ,1)}(x)$; then
for $0<s<1/2$,  $ m_f (s,Q)
=-2$  but $ m_{| f|}( s ,Q) = 1 <  2$.
And, in contrast to averages, the restriction  $0< t < s+s_1- 1$
is necessary for (1.8) to hold. To see this  let $Q = [0, 1]$, and pick
 $1/2< s_1\le s<1$,  and  $t=s+s_1-1>0$.
If $f=\chi_{[0, 1-s]}$ and  $g=\chi_{[1-s_1,2(1-s_1)]}$,
 $\{y\in Q: f(y)+g(y)<1\}=(1-s,1-s_1)\cup(  2(1-s_1),1]$ has
 measure $(s -s_1)+1-2(1-s_1)=t $, and, therefore,
 although $m_f (s,Q) = m_g(s_1, Q) = 0$, $m_{f+g}(t,Q)=1$.

Finally,   maximal medians can be expressed in terms of distribution functions  or
nonincreasing rearrangements.
 Recall that the nonincreasing rearrangement $f^*$
of $f$ at level $\lambda>0$ is given by
$   f^*(\lambda)=\inf \{\alpha>0: |\{x\in \R^n: |f(x)|>\alpha \} | \le \lambda \}$
and satisfies 
\begin{equation}
|\{ y\in \R^n:|f(y)|>f^*(u)\}|\le  u\,,\quad u> 0\,.
\end{equation}

We then have
\begin{proposition} Let $Q\subset \R^n$ be a cube, $0<s<1 $, and $f$ a measurable
function on $Q$. Then
\begin{equation*}
m_{|f|}(1-s,Q)=\inf\{\alpha>0:|\{y\in Q:|f(y)|>\alpha\}|<s|Q|\}= (f\chi_Q)^*(s|Q|)\,.
\end{equation*}
\end{proposition}
\begin{proof}
Let $\overline \alpha =\inf\{\alpha>0:|\{y\in Q:|f(y)|>\alpha\}|<s|Q|\}$.
Then for all $\varepsilon > 0$  it readily follows that
$|\{y \in Q : |f (y) | \le \overline \alpha  + \varepsilon\}| >  (1 - s)|Q|$,
which together with (1.3)   implies
$ m_{|f |} (1-s,Q)\le \overline \alpha  +\varepsilon$. Thus 
  $ m_{|f  |} (1 - s, Q)\le \overline \alpha $.

Next, by (1.12),  $|\{ y\in Q:|f(y)|>(f\chi_Q)^*(s|Q|)\}|\le  s|Q|$, which
gives $\overline \alpha \le  (f\chi_Q)^*(s|Q|)$.

Finally, since  for $\varepsilon>0$, $|\{ y\in Q:|f(y)|>(f\chi_Q)^*(s|Q|)-\varepsilon \}|> s|Q|$,
by (1.2) it readily follows that $(f\chi_Q)^*(s|Q|)-\varepsilon < m_{|f|}(1-s,Q)$, and,
consequently,
$(f\chi_Q)^*(s|Q|)\le m_{|f|}(1-s,Q)$.
\end{proof}

Other equivalent expressions appearing in the literature include those in \cite{CSS}.

\section{Convergence of medians}
 The  examples following Proposition 1.1 suggest that medians rely more heavily on the
distribution of the values of $f$ than do averages.
On the other hand, averages and medians are not always at odds. In particular,
using (1.11) the reader should have no difficulty in  verifying that
the following version of the Lebesgue differentiation theorem holds:
If $f$ is a locally integrable function     on $\R^n$ and $1/2\le  s < 1$,  then
\[\lim_{x\in Q,\, Q\to x} m_f (s,Q) = f(x)\]
at every  Lebesgue point $x$   of $f$.

Thus,  in some sense $m_f (s,Q)$ is a good substitute for $f_Q$ for  small  $Q$.
 In fact, a more careful argument gives that the  biased maximal medians
$m_f(s,Q)$ of an arbitrary measurable function
$f$ converge to $f$ a.e., a fact observed by Fujii for the case $s=1/2$ \cite{F}.

\begin{theorem} Let $f$ be a real-valued, finite a.e.\!  measurable function on $\R^n$, and $0 < s < 1$.
Then
\begin{equation}\lim_{x\in Q,\,Q\to x}m_f (s,Q) = f(x)\text { a.e.}
\end{equation}

In particular, {\rm{(2.1)}} holds at every point of continuity $x$ of $f$.
\end{theorem}

\begin{proof}   For $k\ge 1$ and an integer $j$,  let $E_{k, j} =\{
x\in\R^n: (j- 1)/2^k\le  f(x) < j/2^k\}$,   $a_{k,j}=(j-1)/2^k$, and put
$S_k(x)= \sum_{j=-\infty}^\infty a_{k,j} \chi_{E_{k,j}}(x)$. Note that
since $f$ is finite a.e.,    $\R^n=\bigcup_{k,j} E_{k,j}$ except possibly for a set of measure $0$, and
when $f(x)$ is finite  we have $0\le f(x) -S_k(x)\le  2^{-k}$,
 which gives $ m_{S_k} (s,Q)\le m_f (s,Q)\le m_{S_k}(s,Q)+2^{-k}$ for all cubes $Q$.
Let   $A_{k,j} =\{x \in E_{k,j}: x$ is a point of density for $E_{k,j}\}$, $A_k=\bigcup_{j=-\infty}^\infty A_{k,j}$.
Since $f$ is finite a.e., $|\R^n\setminus A_k|=0$ for all $k$, and if $A=\bigcup_{k=1}^\infty A_k$, also
$|\R^n\setminus A|=0$.

We claim that the limit in question exists for $x\in A$.
Given $\varepsilon>0$, pick $k$ such that $2^{-k+1}<\varepsilon$.
Then  $x\in A_{k,j}$ for some $j$, and
\[\lim_{x\in Q,\,Q\to x} \frac{| A_{k,j}\cap Q |}{|Q|}=1\,.
\]

Let 
$ \delta= \max \{s,1-s\}$  and note that for all
cubes $Q$ with small enough     measure containing $x$,
\[\frac{| A_{k,j} \cap Q|}{|Q|}> \delta\,.
\]

We restrict our attention to such small cubes $Q$ containing $x$.
 Note that for these cubes    $m_{S_k}(s,Q)= a_{k,j}$.
 Indeed, on the one hand, since $ S_k(y)=a_{k,j}$ for $y\in A_{k,j}$,
 $|\{y\in Q: S_k(y)< a_{k,j}\}|\le |A_{k,j}^c\cap Q|< s|Q|$, and,
 therefore, $a_{k,j}\le m_{S_k}(s,Q)$.
  And, on the other,
 since for $\varepsilon>0$, $\{y\in Q: S_k(y)< a_{k,j}+\varepsilon\}\supset A_{k,j}\cap Q$,
it follows that $|\{y\in Q: S_k(y)< a_{k,j}+\varepsilon\}|\ge | A_{k,j}\cap Q|\ge  s |Q|$.
Hence, $m_{S_k}(s,Q)\le a_{k,j}+\varepsilon$, and since $\varepsilon$
 is arbitrary, $m_{S_k}(s, Q)\le a_{k,j}\,$.

Then, since $ a_{k,j}= m_{S_k}(s, Q)=S_{k}(x)$ for $x\in A_{k,j}\,$,
\begin{align*} |m_f(s,Q)- f(x)| &\le  |m_f(s,Q) -m_{S_k}(s,Q)| +|m_{S_k}(s,Q) - f(x)|\\
                                &\le 2^{-k}+ (f(x)- S_k(x))
                                \le 2^{-k+1}<\varepsilon \,.
                                \end{align*}
In other words, $|m_f(s,Q)- f(x)|<\varepsilon$ for $x\in A$ and all $Q$
with small enough measure  containing $x$.

Now, at a point  of continuity $x$ of  $f$,
given $\varepsilon>0$, let  $\delta > 0$ be  such that  $|f(y) -f(x)| \le \varepsilon$ for
$y \in B(x, \delta)$. Then  for $y$ in a cube $Q$ containing $x$ and contained
in $B(x, \delta)$  we have $- \varepsilon\le   f(y) -f(x)   \le \varepsilon$, and,
consequently, $-\varepsilon = m_{-\varepsilon}(s,Q)\le  m_{ f-f(x)}(s,Q)= m_f(s,Q)-f(x)
\le m_\varepsilon(s,Q) =\varepsilon$, and so $| m_f(s,Q)-f(x)|\le \varepsilon$.
\end{proof}

\section{A median characterization for continuity}

 We say that a measurable
function $f$  on a cube $Q_0\subset \R^n$ is equivalent to a continuous function on $Q_0$
if the values of $f$ can be modified on a set of Lebesgue measure $0$ so
as to coincide with a continuous function on $Q_0$; similarly for
$f$  equivalent to a bounded function on a cube.
 In this section we 
characterize  those  measurable  functions equivalent to a bounded function
on a cube that are equivalent to a continuous function on that cube in terms of   medians,
keeping in mind that in the case of locally integrable functions
the condition involves the consideration of oscillations involving two
nonoverlapping cubes \cite{ST}.

\begin{defin}
For $0<s<1$ and nonoverlapping cubes $Q_1,Q_2\subset \R^n$,   let
\[\Psi_s(f, Q_1, Q_2) = \frac{|Q_1|}{|Q_1 \cup Q_2|}\, m_f(s, Q_1) + \frac{|Q_2|}{|Q_1 \cup Q_2|}\, m_f(s,Q_2)\,,\]
and
\[  {\Omega}(f, s,\delta)  =  \sup_{ {\rm diam}(Q_1 \cup Q_2) \le  \delta}\inf_c\,
 \Psi_s (|f - c|, Q_1, Q_2) \,. \]
\end{defin}

 $\Psi_s$ is a weighted average of maximal   medians of $f$ in the spirit of averages and
 ${\Omega}(f, s,\delta) $ is related to the oscillation of a measurable  function
 on a cube, as shown by the following result.


\begin{theorem} Let $Q_0\subset \R^n$ be a cube, $1/2<s<1$,    $f$ a measurable function
on $Q_0$ that is equivalent to a bounded function there, and
 $\omega (f,\delta)$, $\delta > 0$,  the essential modulus of continuity of $f$ defined by
\[\omega (f,\delta)  = \sup_{|h| \le  \delta} \big({\rm ess\ sup}_{x, x+h \in Q_0}|f(x+h) - f(x)|\big) .
\]
Then we have $ \Omega(f, s,\delta)=\omega(f,\delta)/2 $.
\end{theorem}

\begin{proof}
For nonoverlapping cubes $Q_1, Q_2 \subset Q_0$, let
\[\Theta = {\rm ess\ osc} (f, Q_1 \cup Q_2) = {\rm ess\ sup}_{Q_1 \cup Q_2} f - {\rm ess\ inf}_{Q_1 \cup Q_2} f\,.\]
Let $\delta>0$. If $x \in Q_1 \cup Q_2 \subset Q_0$ is such that $x + h \in Q_0$  where $|h|\le \delta$,
  since
\[{\rm ess\ sup}_{x, x+h \in Q_0}|f(x+h) - f(x)|  \ge {\rm ess\ sup}_{Q_1 \cup Q_2}f  - {\rm ess\ inf}_{Q_1 \cup Q_2}f\,,
\]
taking the sup over $|h|\le \delta$ it readily follows that
$\omega (f,\delta) \ge \Theta$.
Moreover, since for $y\in Q_1 $ and an arbitrary constant $c$,
\[|f(y)- c|\le\max \, \{ {\rm ess\ sup}_{Q_1 \cup Q_2} f-c, c-  {\rm ess\ inf}_{Q_1 \cup Q_2} f \},\]
picking $c=({\rm sup\ inf}_{Q_1 \cup Q_2} f + {\rm ess\ inf}_{Q_1 \cup Q_2} f)/2$, it follows
that $|f(y)-c|\le \Theta/2$, and, consequently, $ m_{|f-c|}(s, Q_1)\le  m_{\Theta/2} (s, Q_1)= \Theta/2$;
similarly we have $m_{|f-c|}(s, Q_2)\le \Theta/2$.
Therefore,
\[ \inf_c  \Psi_s( |f - c|, Q_1, Q_2) \le \frac{|Q_1|}{|Q_1 \cup Q_2|}\frac{ \Theta}{2}
+ \frac{|Q_2|}{|Q_1 \cup Q_2|}\frac{\Theta}{2}=\frac{\Theta}{2}\,,\]
and, consequently, $\Omega(f,s, \delta)\le \Theta/2 \le \omega(f,\delta)/2$.

Conversely,  let $0<t<2s-1$.
Then for fixed $\delta>0$,  given $\varepsilon > 0$,
 pick $h$ with $|h| < \delta$ such that
  ${\rm ess\ sup}_{x, x+h \in Q_0}|f(x+h) - f(x)|\ge  \omega(f,\delta) - \varepsilon$.
  Then   $E = \{x\in Q_0: x+h \in Q_0$ and $|f(x+h) - f(x)| \ge \omega(f,\delta) - \varepsilon\}$ has positive measure.
   Let $x\in E$ be a point of density of $E$  and $a$
small enough  so that $Q(x,a), Q(x+h, a)$ are nonoverlapping and
\[\frac{|E \cap Q(x,a)|}{|Q(x,a)|} > 1-t\,.\]
Now, since $\{y\in Q(x+h,a):g(y)<M\}=\{y\in Q(x,a): g(y+h)<M\}$ and $|Q(x,a)|=|Q(x+h, a)|$,
it readily follows that
$m_{|f - c|}(s, Q(x+h,a))=m_{|f (\cdot\, + h)- c|}(s, Q(x,a))$, and, consequently,
since $|f(y+h)-f(y)|\le |f(y+h)-c|+ |f(y)-c|$,
by (1.8) and (1.6),   $m_{|f - c|}(s, Q(x,a))+ m_{|f - c|}(s, Q(x+h,a))
\ge m_{|f-c|+ |f (\cdot\, + h)- c|}(t, Q(x,a))\ge m_{|f(\cdot\,+h)-f|}(t, Q(x,a))$.
Therefore,
\begin{align*}
\Psi_s( |f -& c|, Q(x,a), Q(x+h,a))\\
 &\ge  \frac{1}{2}m_{|f - c|}(s, Q(x, a)) + \frac{1}{2}m_{|f - c|}(s, Q(x+h,a))
\\
&\quad\quad\quad\quad\ge \frac{1}{2} m_{|f(\cdot\, + h) - f |}(t,Q(x,a))\,.
\end{align*}

Finally, since $|E \cap Q(x,a)|=|\{y\in Q(x,a): |f(y + h) - f(y) | \ge \omega(f,\delta) - \varepsilon\}|>
(1-t)|Q(x,a)|$,   it readily follows that $ m_{|f(\cdot\, + h) - f |}(t,Q(x,a))\ge
  \omega(f,\delta) - \varepsilon$, which, since $\varepsilon$ is arbitrary, implies
$\Psi_s( |f - c|, Q(x,a), Q(x+h,a))\ge  \omega(f,\delta)/2$.
Thus $\Omega(f,s, \delta + (\sqrt{2a})^n) \ge \omega(f,\delta) $, and letting $a \to 0$,
 $\Omega(f,s,\delta) \ge  \omega(f,\delta)/2$.
\end{proof}

\begin{theorem} Let $Q_0\subset \R^n$ be a cube, $1/2<s<1$, and  $f$ a measurable function
on $Q_0$ that is equivalent to a bounded function there.
 Then, $f$ is equivalent to a continuous function on   $Q_0$
 iff   $ \  \lim_{\eta \to 0^+} \Omega (f,s,\eta) = 0$.
\end{theorem}

The proof  follows at once from Theorem 3.1.  Note that by Proposition 1.2
the conclusion can also be stated in terms
of rearrangements.

\section{A decomposition of cubes}

  Str\"{o}mberg's  essential tool in dealing with   the oscillation of functions
and local maximal functions is a decomposition of cubes \cite{JOS}.
 In this section we extend  the  results
to biased medians with parameters $> {1}/{2}$.

We begin by introducing the local sharp maximal function restricted
to a cube.

\begin{defin}
Let $Q_0\subset \R^n$ and  $0 < s \le {1}/{2}$. For a  measurable function $f$  on $Q_0$,
$M_{0,s,Q_0}^{\sharp}f(x)$,  the  \emph{ local sharp maximal function  restricted to  $Q_0$ of $f$}
is defined at $x\in Q_0$ as
\begin{equation}
 M_{0,s,Q_0}^{\sharp}f(x) = \sup_{ x\in Q, Q\subset Q_0 }
 \inf_c \inf\{\alpha \ge 0: |\{y \in Q: |f(y) - c| > \alpha\}| < s|Q| \}\,.
 \end{equation}

 When $Q_0=\R^n$,
$M_{0,s,\R^n}^{\sharp} f(x)= M_{0,s}^{\sharp} f(x)$ denotes the \emph{local sharp maximal function of $f$
at $x\in\R^n$}.
\end{defin}

The range $0 < s\le 1/2$ is necessary  since for $s > 1/2$,
 $M^\sharp_{0,s,Q} f (x) = 0$ for a function
$f$ that takes two different values.

Local maximal functions, as well
as maximal functions defined in terms of rearrangements,
 can be expressed in terms of medians. Let
$\omega_s(f,Q)= \inf_c ((f-c)\chi_Q)^*(s|Q|)$. Then
by Proposition 1.2,
\[M_{0,s,Q_0}^\sharp f (x) =
 \sup_{x\in Q, Q\subset Q_0}\omega_s(f,Q)=\sup_{x\in Q, Q\subset Q_0} \inf_c m_{|f -c|} (1 - s, Q)\,.\]
The first expression above is used by Lerner \cite{L, La}.

An efficient choice  for $c$ in the infimum above is $ m_{|f - m_f(1-s,Q)|}(1-s,Q)$.
Indeed,   for $Q \subset Q_0$ and a constant $c$, since $ 1-s \ge 1/2$, by (1.7)  and (1.10),
\begin{equation} |  m_f(1-s,Q)-c|\le   m_{|f-c|}(1-s,Q) \,.
\end{equation}
Then,  since $ |f(y) - m_f(1-s,Q)|\le |f(y) - c| + |c - m_f(1-s,Q)|$, by (1.5), (1.7),
and (4.2),
\begin{align*}
m_{|f -  m_f(1-s,Q)|}&(1-s,Q)  \le
   m_{|f - c|}(1-s,Q) + |c - m_f(1-s,Q)|\\
   & \le  m_{|f - c|}(1-s,Q) + m_{|f - c|}(1-s,Q)
 = 2\, m_{|f - c|}(1-s,Q)\,,
\end{align*}
and, consequently,
\begin{equation} \inf_c m_{|f -c|} (1 - s, Q) \le   m_{|f - m_f(1-s,Q)|}(1-s,Q) \le 2\,
 \inf_c m_{|f -c|} (1 - s, Q)\,.
\end{equation}

The decomposition of cubes relies on three lemmas which we prove next.
\begin{lemma}
Let $Q\subset \R^n$ be a cube,  $0 < s \le  {1}/{2}$, ${1}/{2} \le t \le 1-s$,
and $f$
a  measurable function on $Q$. 
Then for   any $\eta>0$,
\begin{equation}
| \{y \in Q: |f(y) - m_f(t,Q) | \ge  2\inf_{x\in Q} M^{\sharp}_{0,s,Q }f(x)+\eta \}| <  s|Q|\,.
\end{equation}
\end{lemma}

\begin{proof}
For fixed $c$, let $\alpha(c)=m_{|f-c|}(1-s,Q)$.
Then by (4.2) and (1.5),
\begin{equation}
 |m_f (t, Q) - c| \le  m_{|f -c|} (t, Q) \le  m_{|f -c|} (1 - s, Q) =  \alpha (c)\,,
 \end{equation}
and by (1.4)
\begin{equation} |\{y \in Q : |f (y) - c| \ge \alpha(c) + \varepsilon\}| < s|Q|\,,\quad \varepsilon > 0\,.
\end{equation}

Let $m=\inf_c \alpha(c)$ and pick $\{c_k\} $ such that
$m\le \alpha(c_k)\le m+1/k$, all $k$.
 Then  by (4.5), 
\begin{align*}
|f(y) -c_k| &\ge |f(y)-  m_f(t,Q)|-| m_f(t,Q)-c_k|\\
&\ge |f(y)-  m_f(t,Q)|-\alpha(c_k)\,,
\end{align*}
 and, consequently,  since $2m+\eta\ge \alpha(c_k)+ (\eta-2/k)$,
$\{y\in Q:  |f(y)  -m_f(t,Q)| \ge 2m+\eta  \}\subset
 \{y\in Q:  |f(y) -c_k|\ge    \alpha(c_k) +  \varepsilon_k\}$, where we have chosen $k$ sufficiently large so that
 $\varepsilon_k=\eta-(2/k)>0$.
Then by (4.6), $ |\{y\in Q:  |f(y) -m_f(t,Q)| > 2 m+\eta  \}|<s|Q|$.
Finally, since $M^{\#}_{0,s,Q}f(x) \ge m$ for all $x \in Q$, (4.4) holds.
\end{proof}

\begin{lemma}
Let $Q\subset \R^n$ be a cube,  $0 < s \le  {1}/{2}$, ${1}/{2} \le t \le 1-s$,
$\eta>0$, and $f$ a measurable function on an open cube containing $Q$.
Then for any family of cubes  $\{ Q_{\varepsilon}\}$ with $(1-\varepsilon)Q \subset Q_{\varepsilon}
\subset (1+\varepsilon)Q$,
\[\limsup_{\varepsilon \to 0^+} |m_f(t,Q) - m_f(t, Q_{\varepsilon})| \le  2 \inf_{x\in Q} M^{\sharp}_{0,s,Q}f(x)+\eta\,.\]
\end{lemma}

\begin{proof}
Let $A=\inf_{x\in Q} M^{\sharp}_{0,s,Q}f(x)$.
For the sake of argument assume
  there is a sequence $\varepsilon_k\to 0$ such that
$|m_f(t,Q) - m_f(t, Q_{\varepsilon_k})| >   2A+\eta  $ for all $k$.
Then by   (1.10),
\[  2A+\eta < |m_f(t,Q) - m_f(t, Q_{\varepsilon_k})|\le  m_{|f - m_f(t, Q)|}( t, Q_{\varepsilon_k})  \,,\]
and, consequently, by (1.4),
\begin{equation}
|\{y\in Q_{\varepsilon_k}: |f(y)-m_f(t,Q)| >  2A+\eta \}| \ge (1- t)|Q_{\varepsilon_k}|\,.
\end{equation}

Since $Q_{\varepsilon_k}\subset (1+\varepsilon_k)Q$, the left-hand side of (4.7) is bounded above by
 \begin{align*} |\{y\in &(1+\varepsilon_k) Q  : |f(y)-m_f(t,Q)| >  2A+\eta\}|\\
 & \le  ((1+\varepsilon_k)^n-1)|Q|+ |\{y\in Q: |f(y)-m_f(t,Q)| >  2A+\eta\}|\,,
  \end{align*}
and since $ (1-\varepsilon_k)Q\subset Q_{\varepsilon_k}$, the right-hand side of (4.7) is bounded below by
\[ (1-t)( 1-\varepsilon_k)^n |Q|\,.
\]
Whence combining these estimates it follows that
\[|\{y\in Q: |f(y)-m_f(t,Q)| > 2A+\eta\}|\ge \big((1-t)(1-\varepsilon_k)^n  -((1+\varepsilon_k)^n-1)\big)|Q|\,.\]

 Now, by (4.4) there exists $\delta>0$ such that
$\big| \{y \in Q: |f(y) - m_f(t,Q) | \ge  2A+\eta \}\big|=( s-\delta)|Q|$, and so
\[  ( s-\delta)  \ge \big((1-t)(1-\varepsilon_k)^n  -((1+\varepsilon_k)^n-1)\big)\,.
\]
Thus,  letting $k\to \infty$,
$ s-\delta\ge  1-t\ge s$,  which is not the case.

\end{proof}

\begin{lemma}    Let $Q_0, Q_1\subset \R^n $ be cubes with
$Q_0 \subset Q_1$ and  $|Q_1| \le 2^k |Q_0|$ for some integer $k$,
 $0 < s \le  {1}/{2}$, ${1}/{2} \le t \le 1-s$, and $f$ a measurable function on $Q_1$.
Then
\begin{equation} |m_f(t,Q_0) - m_f(t,Q_1)| \le  10k\inf_{x\in Q_0} M^{\sharp}_{0,s,Q_1}f(x)\,.
\end{equation}
\end{lemma}

\begin{proof}
By the triangle inequality it suffices to prove the case $k=1$. Let $A = \inf_{x\in Q_0} M^{\sharp}_{0,s,Q_1}f(x)$.
For the sake of argument suppose that  (4.8) does not hold. Then if
$A>0$, by Lemma 4.2,
for any fixed $0<\eta<A/2$,
there exists a cube  $Q_2 $  such that $Q_0 \subset Q_2 \subset Q_1$ and
\[ |m_f(t,Q_2) - m_f(t,Q_0)| > 4A+ 2\eta\,, \quad |m_f(t,Q_2) - m_f(t,Q_1)| > 4A+2\eta\,.\]
And if $A=0$, then $ |m_f(t,Q_0) - m_f(t,Q_1)|>0 $ and
there exists a cube  $Q_2 $  such that $Q_0 \subset Q_2 \subset Q_1$ and
\[ |m_f(t,Q_2) - m_f(t,Q_0)| >  2\eta\,, \quad |m_f(t,Q_2) - m_f(t,Q_1)| > 2\eta\]
for $\eta$ sufficiently small.

Thus in both cases the sets $\{y \in Q_k: |f(y) - m_f(t,Q_k)| \le
  2A+\eta\}$, $k = 0,1,2$, are pairwise disjoint subsets of $Q_1$, and, consequently,
\begin{align*} \{y \in Q_0:& |f(y)  - m_f(t,Q_0)| \le  2A+\eta\}\\
& \cup \{y \in Q_2: |f(y) - m_f(t,Q_2)| \le 2A+\eta\}\\
& \quad\quad\quad\quad \subset \{y \in Q_1: |f(y) - m_f(t,Q_1)| > 2A+\eta\}\,.
\end{align*}
Therefore,  since $\inf_{x\in Q_k} M^{\sharp}_{0,s,Q_k}f(x)\le A$ for $k=0,1,2$,
 by Lemma 4.1,
\[ (1-s)|Q_0|+ (1-s)|Q_2|<s|Q_1|\,.
\]
Thus $2(1-s)|Q_0| < s|Q_1| <2 s|Q_0|$, and, consequently,
   $1 < 2s$, which is not the case.
\end{proof}

We are now ready to consider the decomposition of cubes relative to medians.
\begin{proposition}
Let $Q\subset \R^n$ be a cube,  $0 < s \le {1}/{2}$, ${1}/{2} \le t \le 1 - s$, $\delta,\beta > 0$, and $f$ a measurable
function on $Q$.   Then if $|m_f(t,Q)| \le  \delta$, there exists a (possibly empty) family of
nonoverlapping dyadic subcubes $\{Q_k\}$ of $Q$ so that
\begin{enumerate}
\item[\rm (1)]
$Q_k \not \subset \{y \in Q: M^{\sharp}_{0,s,Q}f(y) > \beta\}$,
\item[\rm (2)]
$\delta < |m_f(t,Q_k)| \le  \delta + 10\,n\,\beta$,
\item[\rm (3)]
$|f(x)| \le  \delta$ for a.e.\! $x \in Q \setminus \big( \{y \in Q: M^{\#}_{0,s,Q}f(y) > \beta\} \cup \bigcup_k Q_k \big)$.
\end{enumerate}
\end{proposition}

\begin{proof}
If $ M^{\sharp}_{0,s,Q}f(y) > \beta$ for all $y\in Q$
we pick  $\{Q_k\}$ as the empty family.
Otherwise subdivide $Q$ dyadically into $2^n$ subcubes and note that by Lemma 4.3
    for each dyadic subcube $Q'$,
  $|m_f (t,Q')|\le \delta  + 10\,n \inf_{x\in Q'} M^\sharp_{0,s,Q}f(x)$.
Thus for each of these subcubes $Q'$   one of the following holds:
\begin{enumerate}
\item[\rm (a)]
$Q' \subset \{y \in Q: M^{\sharp}_{0,s,Q}f(y) > \beta\}$: we discard $Q'$.
\item [\rm (b)]
$Q'$ satisfies conditions (1) and (2) above: we collect this $Q'$.
\item [\rm (c)]
$Q' \not \subset \{y \in Q: M^{\sharp}_{0,s,Q}f(y) > \beta\}$ but $|m_f(t,Q')| \le  \delta$: we subdivide $Q'$ and continue in this fashion.
\end{enumerate}

Finally, a.e.\! $x \in Q \setminus \big(  \{y \in Q: M^{\sharp}_{0,s,Q}f(y) > \beta\}\cup \bigcup_k Q_k \big)$
is contained in arbitrarily small cubes $\{Q_k(x)\}$ containing $x$
so that $|m_f(t,Q_k(x))| \le \delta$. By Theorem 2.1
it readily follows  that $|f(x)| \le \delta$ for a.e.\! such $x$.
\end{proof}

We can be more precise in the description of the cubes above when $M^{\sharp}_{0,s,Q}f \in L^\infty(Q)$. Observe that
$f-m_f(t,Q)$ satisfies $m_{f-m_f(t,Q)}(t, Q)=0$ and $M^{\sharp}_{0,s,Q}f(x)=M^{\sharp}_{0,s,Q}(f-m_t(t,Q))(x)$ for all $x\in Q$,
which means that the decomposition for $f-m_f(t,Q)$ holds for %
any $\delta>0$.
Let $\{Q_j\}$ and $\{Q_k\}$ denote the families of cubes obtained from the decomposition with parameters
$\beta\ge \|M^{\sharp}_{0,s,Q}f\|_{L^\infty(Q)}$ and  $\delta_1= 4\beta  +2\eta$ and $\delta_2=2\delta_1 + 10n\beta $, respectively.

Observe that by construction  we have $\bigcup_k Q_k\subset \bigcup_j Q_j$.
To see this consider   a dyadic
subcube $Q' $ of $Q$ that has not been discarded; this depends on
$\beta$ and not  on   $\delta_1$ or $\delta_2$. If
$Q'$ is a $Q_j$, then
 $\delta_1 < |m_{f-m_f (t,Q)}(t,Q')|\le \delta_1 + 10n\beta < \delta_2$ and $ Q'$ is not a $Q_k$.
 So any  $Q_k$  contained in $Q'$ arises
 from subsequent subdivisions of   $Q'$. On the other hand, if $Q'$ is not a $Q_j$,
 then $|m_{f-m_f (t,Q)}(t,Q')|\le \delta_1 < \delta_2$ and $Q'$ is not a $Q_k$ either. Since this relation is maintained
 at every level of the successive dyadic subdivisions, the $Q_k$'s arise from subdivisions of $Q_j$'s.

From here on  the argument proceeds as in Lemma 4.3. Note that
\[\delta_1 < |m_{f -m_f (t,Q)} (t, Q_j )|= |m_f (t, Q_j ) - m_f (t, Q)|
\le \delta_1+10 n\beta =\delta_2-\delta_1\,,\]
 and
 \[  \delta_2 < |m_{f -m_f (t,Q)} (t, Q_k )| = |m_f (t, Q) -
m_f (t, Q_k )|\,.\]
Therefore,
\begin{align*}|m_f (t, Q_j ) - m_f (t, Q_k )| &\ge |m_f (t, Q_k ) - m_f (t, Q)| - |m_f (t, Q_j ) - m_f (t, Q)|\\
&> \delta_2-(\delta_2 - \delta_1)= \delta_1 =4\beta+2\eta\,.
\end{align*}

Thus it readily follows that the sets $\{y \in Q : |f (y)-m_f (t, Q)| \le 2\beta +\eta\}$,  $\{y \in Q_j : |f (y)-m_f (t, Q_j )| \le
2\beta + \eta\}$, and $\{y\in Q_k : |f (y) - m_f (t, Q_k )| \le 2\beta + \eta\}$ are nonoverlapping, and so
\begin{align} \{y \in Q_j:& |f(y)  - m_f(t,Q_j)| \le  2\beta+\eta\}\\
& \cup \{y \in Q_k: |f(y) - m_f(t,Q_k)| \le 2\beta+\eta\}\nonumber\\
& \quad\quad\quad\quad \subset \{y \in Q: |f(y) - m_f(t,Q)| > 2\beta+\eta\}\nonumber\,.
\end{align}

Now, since  $ \inf_{ x\in Q_j} M_{0,s,Q_j}^\sharp f (x) \le\beta$ for each $Q_j$,
by (4.4) it follows that
 \begin{equation*}
 (1-s) \sum_j |Q_j|\le   \sum_j |  \{y \in Q_j: |f(y)  - m_f(t,Q_j)| \le  2\beta+\eta\}|\,,
 \end{equation*}
and a similar estimate holds with the $Q_k$'s in place of the $Q_j$'s. Finally, since the
 sets in the left-hand side of (4.9) are pairwise disjoint for all $j$ and $k$,
and since $\bigcup_k Q_k\subset \bigcup_j Q_j$ and $ \inf_{ x\in Q} M_{0,s,Q}^\sharp f (x) \le\beta$,
by Lemma 4.1,
\[2\sum_k|Q_k|\le \sum_k|Q_k|+\sum_j|Q_j|
\le \frac{s}{1-s}|Q|\,,
\]
and, consequently, since $2(1-s)\ge 1$,
\begin{equation} \sum_k|Q_k|\le \frac{s}{2(1-s)} |Q|\le  s  |Q|\,.
\end{equation}

\section{Vanishing median oscillation}

We say that a measurable function $f$
defined on a cube $Q_0\subset\R^n$ is of vanishing median oscillation with
parameter $s$ in $Q_0$ $(vmo_s(Q_0) )$ if
\[ \phi_s(u) = \sup_{Q \subset Q_0, |Q|\le u} \inf_c  m_{|f - c|}(1-s,Q) \]
satisfies $\lim_{u\to 0^+}\phi_s(u)=0$.

Note that by (1.11),
\[  \inf_c m_{|f-c|}(1-s, Q)\le \frac1{ s}\frac1{|Q|} \int_Q|f(y)-f_Q|\,dy\,,
\]
and, therefore,  $\lim_{u\to 0^+}\phi_s(u)=0$ for all $s$ whenever $f\in  VMO(Q_0)$.
Here we show that the spaces actually coincide for $s\le 2^{-n}$.

Now, $\phi_s$ is a nonnegative, nondecreasing continuous function
that vanishes at the origin, and   $vmo_s$ may be described in terms
of such functions  $\phi$ as follows. Let
\[||f||_{s, \phi,Q_0} = \sup_{Q \subset Q_0} \inf_c \frac{m_{|f - c|}(1-s,Q)}{\phi(|Q|)}\sim
\sup_{ Q \subset Q_0} \frac{m_{|f - m_f(1-s,Q)|}(1-s,Q)}{\phi(|Q|)}\,,\]
and $bmo_{s,\phi}(Q_0) =\{f: f$ is defined and measurable on $Q_0$, and $||f||_{s,\phi,Q_0}<\infty\}$.
Then $ vmo_s(Q_0)=  \bigcup_{\phi} bmo_{s,\phi}(Q_0)$.

Now fix $Q_0$ and $0 < s \le  {1}/{2^n}$. Let $\phi: \R^+ \to \R^+$ be continuous, nondecreasing, and
$\phi(0) = 0$, and define $\Psi_{|Q_0|}:[0, 2^n|Q_0|]\to \R^+$ by
\begin{equation}
\Psi_{|Q_0|}(u) = \int_{u}^{2^n|Q_0|} \frac{\phi(v)}{v}\, dv\,.
\end{equation}

We can then prove a strengthened version of the John-Nirenberg inequality.

\begin{theorem}
Let $f \in bmo_{s,\phi}(Q_0)$ for some $\phi$ as above and
$\Psi_{|Q_0|}(u)$   given by $\rm {(5.1)}$. Then  there exist constants $c_1, c_2$
independent of $f$ and all subcubes $Q\subset Q_0$ so that
\begin{equation}
|\{y \in Q: |f(y) - m_f(1-s,Q)| > \lambda\}| \le  c_1 \Psi_{|Q|}^{-1} \big(  {c_2 \lambda}/{\|f\|_{s,\phi, Q}} \big ),
\quad \lambda>0\,.
\end{equation}
\end{theorem}

\begin{proof} If $\|f\|_{s,\phi,Q_0}=0$,
clearly $M_{0,s,Q_0}^\sharp f (x)=0$ for all $x\in Q_0$ and by Lemma 4.3,  the medians of
$f$ over all subcubes of $Q_0$ are constant. Then  by Theorem 2.1, $f$ is a.e.\! constant, and the conclusion
 holds in this case.  Otherwise,
since $\|f-c\|_{s,\phi,Q_0}=\|f\|_{s,\phi,Q_0}$  and $\|c f\|_{s.\phi,Q_0} = |c|\, \|f\|_{s,\phi,Q_0}$ for all constants $c$,
we may assume that $\|f\|_{s, \phi,Q_0} = 1$ and $m_f(1-s,Q_0) = 0$.
Then   by (4.3),
\[m_{|f - m_f(1-s,Q)|}(t,Q) \le 2\phi(|Q|) \le  2 \phi(|Q_0|)\quad  {\text {for all }} Q \subset Q_0\,,\]
and, consequently,   $\|M^{\sharp}_{0,s,Q_0}f\|_{L^{\infty}(Q_0)} \le 2\phi(|Q_0|)$.
Pick now $\beta_0 = 2\phi(|Q_0|)$, and note that since $\phi(u)>0$ for $u>0$,
$\delta_0 = (10n + 9)\beta_0$ works in Proposition 4.1 and in the comments that follow it.
  Since $|\{y \in Q_0: M^{\sharp}_{0,s,Q_0}f(y) > \beta_0\}| = 0$
we get a (first-generation)   family $\{Q_j^{ 1 }\}$ of nonoverlapping subcubes of $Q_0$ so that
\begin{enumerate}
\item
$\delta_0 < |m_f(t,Q_j^{ 1})| \le \delta_0+  10\,n\beta_0$ for all $j$,
\item
$|f(x)| \le  \delta_0$ for a.e.\! $x \in Q_0 \setminus \bigcup_j Q_j^{1}$, and
\item
$\sum_j |Q_j^{1}| \le  s|Q_0|$.
\end{enumerate}

Now we fix one cube $Q_j^{ 1}$ of this family, which for simplicity we denote
  $ Q^{1}$, and define $g = f - m_f(t,Q^{1})$. Note that $m_g(t,Q^{1}) = 0$
 and $g - m_g(t,Q) = f - m_f(t,Q)$ for all $Q\subset Q^1$. Then as above  we have that
 $m_{|g - m_g(t,Q)|}(t,Q) \le  2\phi(|Q|)$ for all $Q\subset Q^1$, and thus
 $\|M^{\sharp}_{0,s,Q^{ 1 }}g||_{L^{\infty}(Q^{ 1 })} \le  2\phi(|Q|) =  2\phi ({|Q_0|}/{2^n}  )$.

 We pick then (first-generation) parameters $\beta_1 = 2 \phi ({|Q_0|}/{2^n}  )$ and
  $\delta_1 = (10n +9)\beta_1$, which gives
 $|\{y \in Q^{1}: M^{\sharp}_{0,s,Q^{1}}g(y) > \beta_1\}| = 0$. As before
 we get a (second-generation) nonoverlapping  family $\{Q_j^{2}\} \subset Q^{1}$ so that
\begin{enumerate}
\item
$\delta_1 < |m_g(t,Q_j^{2})| \le \delta_1 + 10n\beta_1$ for all $j$,
\item
$|g(x)| \le \delta_1$ for a.e.\! $x \in Q^{1} \setminus \bigcup_j Q^{2}_j$, and
\item
$\sum_j |Q^{2}_j| \le  s |Q^{1}|$.
\end{enumerate}

We can keep control of  the cubes we are gathering and $f$. Indeed, clearly
\[\sum_{ j} |Q_j^2|\le s\sum_{k}|Q_k^1|\le s^2|Q_0|\,.\]
And as for $f$  we have that for a.e.\!  $x \in Q^{ 1} \setminus \bigcup_j Q^{2}_j$,
\begin{align*} |f(x)| &\le  |f(x) - m_f(t,Q^{1})| + |m_f(t,Q^{1}| = |g(x)| + |m_f(t,Q^{1})|\\
&\le \delta_1 + \delta_0 + 10\,n\beta_0 \le (20n + 9)(\beta_0 + \beta_1)\,.
\end{align*}

Continuing in this fashion, the computation becomes clear:
Having selected the $(k-1)$st generation of subcubes $\{Q^{k-1}\}$, we then select a $k$th generation of subcubes so that
\begin{enumerate}
\item
 With $\beta_j = 2\phi(|Q_0|/2^{nj})$ and $\delta_j = (10n + 9)\beta_j$,
 $0\le j\le k-1$, we have
 $|f(x)| \le  ( 10 n+9)\sum_{j=0}^{k-1}\delta_j  +10n\sum^{k-2}_{j=0}\beta_j\le (20n+9)\sum^{k-1}_{j=0}\beta_j$
   for a.e.\! $x \in Q^{k-1} \setminus
\bigcup_j Q^{k}_j$,  and
\item
$\sum_j |Q^{k}_j| \le   s^k|Q_0|$.
\end{enumerate}

This is all that is needed. Suppose first that $\lim_{u \to 0^+} \Psi_{|Q_0|}(u) =  \infty$.
Then, for   $\lambda > 2 (20n+9) \phi(|Q_0|)$,
let $k$   be the largest integer so that $2(20n+9) \sum_{j=0}^{k-1} \phi ({|Q_0|}/{2^{nj}}) < \lambda$ and observe that
by (1) above $\{y \in Q_0: |f(y)| > \lambda\} \subset \bigcup_j Q^{k}_{j}$,
and so by (2) above $|\{y \in Q_0: |f(y)| > \lambda\}| \le  s^k|Q_0|$.
Furthermore, by this choice of $k$ we have
\begin{align*}\lambda  &< 2(20 n+9)\sum_{j=0}^k \phi ( {|Q_0|}/{2^{nj}} ) \\
&\le \frac{2( 20n+9)}{n \ln(2)} \int^{2^n|Q_0|}_{ {|Q_0|}/{2^{kn}}} \phi(u) \frac{du}{u} = c \Psi_{|Q_0|} \Big( \frac{|Q_0|}{2^{kn}} \Big)
\end{align*}
where $c= 2 (20n+9)/ n \ln(2)$.

 So,
 \[\frac{|Q_0|}{2^{kn}} \le  \Psi_{|Q_0|}^{-1}(c_2 \lambda)\,,\quad
 c_2 = \frac{n \ln( 2)}{2(4 + 10n)}\,.\]
 Finally,
 \[  |\{y \in Q_0: |f(y)| > \lambda\}| \le  s^k|Q_0|\le \frac{|Q_0|}{2^{kn}} \le  \Psi_{|Q_0|}^{-1}(c_2 \lambda)\,, \]
 as we wanted to show.

And, for $\lambda \le 2 (20n+9) \phi(|Q_0|)$, pick $c_1$ so that $|Q_0| \le  c_1 \Psi_{|Q_0|}^{-1}(c_2 \lambda)$, and
since $\{ y\in Q_0: |f(y)|>\lambda \}\subset Q_0$, the conclusion holds. Clearly the argument works for all $Q\subset Q_0$.

Finally, in case $\lim_{u \to 0^+} \Psi_{|Q_0|}(u) < \infty$,  the above argument works
for all integers $k$ and, therefore, $f$ is an essentially bounded function on $Q_0$
that satisfies (5.2).  A more thorough argument shows that $f$ is
equivalent to an  essentially Lipschitz function on $Q_0$  \cite {SP}.
\end{proof}

It is now straightforward to  verify that if $f\in vmo_s(Q_0)$,
$f\in VMO(Q_0)$. Pick $\phi$ such that $f\in  bmo_{s,\phi, Q_0}$. Then  integrating
(5.2)  with respect to $\lambda$ it follows that for all subcubes $Q\subset Q_0$,
\begin{align*} \int_Q
|f (y) - m_f (1-s, Q)| \, dy &=\int_0^\infty |\{y \in Q: |f(y) - m_f(1-s,Q)| > \lambda\}|\,d\lambda\\
&\le  c_1  \int_0^\infty \Psi_{|Q|}^{-1} \big(  {c_2 \lambda}/{\|f\|_{s,\phi, Q_0}} \big )\,d\lambda\\
&\le c {\|f \|_{s,\phi, Q_0}} |Q|\phi(2^n|Q|)\,.
\end{align*}

Therefore,
\[\sup_{Q\subset Q_0,|Q|\le u} \frac1{|Q|} \int_Q
|f (y) -  f_Q| \, dy\le c\,\phi(2^n u)\to 0
\]
as $u\to 0^+$ and $f\in VMO(Q_0)$.

\end{document}